\documentclass[3p]{article}
\usepackage{amsfonts,amsmath,amssymb,amsthm,stmaryrd,hyperref}
\usepackage[francais,english]{babel}
\usepackage[T1]{fontenc}
\usepackage[ansinew]{inputenc}  
\usepackage{graphics}
\usepackage{graphicx}
\usepackage{tikz}
\usepackage{multicol}
\usetikzlibrary{patterns}
\usepackage{amssymb,subfigure,lineno,url}
\usepackage{float}

\newtheorem{theorem}{Theorem}
\newtheorem{corollary}{Corollary}
\newtheorem{definition}{Definition}
\newtheorem{lemma}{Lemma}
\newtheorem{proposition}{Proposition}
\newtheorem{rmk}{ Remark}

\begin{document}
\title{Some nonexistence results for space-time fractional Schr{\"o}dinger equations without gauge invariance}

\author{Mokhtar Kirane\footnote{\noindent 
Department of Mathematics, Faculty of Arts and Science, Khalifa University, P.O. Box: 127788,  Abu Dhabi, UAE; mokhtar.kirane@ku.ac.ae
 \newline \indent $\,\,{}^{1}$  Department of Mathematics, Sultan Qaboos University,  FracDiff Research Group (DR/RG/03),  P.O. Box 46, Al-Khoud 123, Muscat, Oman; ahmad.fino01@gmail.com; a.fino@squ.edu.om}, Ahmad Z. Fino$^{\,1}$}

\date{}
\maketitle

\begin{abstract}
In this paper, we consider the Cauchy problem in $\mathbb{R}^N$, $N\geq1$, for semi-linear Schrödinger equations with space-time fractional derivatives. We discuss the nonexistence of global $L^1$ or $L^2$ weak solutions in the subcritical and critical cases under some conditions on the initial data and the nonlinear term. Furthermore, the nonexistence of local $L^1$ or $L^2$ weak solutions in the supercritical case are studied.
\end{abstract}

\maketitle

\noindent {\small {\bf MSC[2020]:} 35A01, 26A33} 

\noindent {\small {\bf Keywords:} Schr{\"o}dinger equations, Fractional derivatives and integrals, test function method, nonexistence of global solution}

\section{Introduction}

In this paper, we consider the problem 
\begin{equation}\label{1}
\left\{
\begin{array}{ll}
\,\,\displaystyle{i^\alpha\, {}^cD^\alpha_{0|t}u-(-\Delta)^{\beta/2} u=\lambda |u|^p,}&\displaystyle {(t,x)\in (0,T)\times\mathbb{R}^{N},}\\
\\
\displaystyle{u(x,0)=\varepsilon u_0(x)},&\displaystyle{x\in {\mathbb{R}^N}, }
\end{array}
\right. \end{equation}
where $u$ is a complex-valued unknown function of $(t,x)$, $0<\alpha<1$, $0<\beta<2$, $N\geq1$, $p>1$, $T>0$, $\varepsilon>0$, $\lambda\in\mathbb{C}\setminus\{0\}$, and $i^\alpha$ is the principal part of $i^\alpha$, i.e.
$$i^\alpha=\cos(\frac{\alpha\pi}{2})+i\cos(\frac{\alpha\pi}{2}),$$
${}^cD^\alpha_{0|t}$ is the Caputo fractional derivative and $(-\Delta)^{\beta/2}:L^2(\mathbb{R}^N)\rightarrow L^2(\mathbb{R}^N)$ is the fractional Laplacian which can be defined by a pointwise representation as given in Definition \ref{def1} below. \\

Different fractional generalizations of the Schr{\"o}dinger equation appeared in the literature: The classical Schr{\"o}dinger equations with nongauge power nonlinearity, i.e. \eqref{1} with $\alpha=1$ and $\beta=2$, has been studied by Ikeda and Wakasugi \cite{Ikeda1} and Ikeda and Inui \cite{Ikeda2,Ikeda3}, the spatial fractional
Schr{\"o}dinger  equation involving fractional order space derivatives,  i.e. \eqref{1} with $\alpha=1$ and $\beta\in(0,2)$, has been investigated in \cite{Fino1,Fino2,Laskin1,Laskin2,Laskin3},
the fractional temporal Schr{\"o}dinger  equation involving a fractional time derivative, i.e. $\alpha\in(0,1)$ and $\beta=2$, has been studied in \cite{Zhangsun,Naber,Narahari}, the semirelativistic Schr{\"o}dinger equation with nongauge invariant power nonlinearity, i.e. \eqref{1} with $\alpha=1$ and $\beta=1/2$, got interest by Fujiwara \cite{Fujiwara1}, Inui \cite{Inui}, Fujiwara and Ozawa \cite{Fujiwara2}, and the spatio-temporal fractional Schr{\"o}dinger equation with both
time and space fractional derivatives attracted the attention of \cite{Dong,Saxena}.\\

The expected critical exponent can be determined by the following scaling argument: If $u(x,t)$ is a solution of (\ref{1}) with initial data $u_0$, then 
$$v(t,x)=\gamma^{\frac{\beta}{p-1}}u(\gamma^{\beta/\alpha} t,\gamma x),$$
 for all $\gamma>0$, is also a solution of (\ref{1}) with initial data $v_0(x)=v(0,x)=\gamma^{\frac{\beta}{p-1}}u_0(\gamma x)$, for all $x\in\mathbb{R}^N$. We choose $p=p_s$ such that we get an invariant $H^s$-norm of the initial data:
$$\left\|v_0\right\|_{H^{s}}=\gamma^{\frac{\beta}{p-1}-\frac{N-2s}{2}}\|u_0\|_{H^{s}}=\|u_0\|_{H^{s}};$$
this happens if and only if 
$$p=p_s=1+\frac{2\beta}{N-2s}.$$
Therefore, the case $p=p_s$ is called $H^s$-critical case; the case $p<p_s$ (resp. $p>p_s$) is called $H^s$-subcritical case (resp. $H^s$-supercritical case). On the other hand, the Fujita critical exponent for the corresponding heat equation with fractional Laplacian is
$$p_F=1+\frac{\beta}{N}.$$
Our main goal is to study the nonexistence of $L^1$ or $L^2$ global weak solutions under the condition $p\leq p_s$ or $p\leq p_F$ as well as the nonexistence of $L^1$ or $L^2$ local weak solutions under the condition that $p> p_s$ or $p> p_F$ (see e.g. \cite{Ikeda3}), using the test function method (see e.g. \cite{Zhang}) or a fractional differential equation approach (i.e. construct a fractional differential equation for a new function and using comparison principle).  The local existence for \eqref{1} is expected in the $H^s$-subcritical case, but this is not our case. We refer the reader to \cite[Appendix]{Ikeda1} by using the Strichartz estimates recently studied by Lee \cite{Lee}.\\
 Let
$$X_T=\{\varphi\in C([0,\infty),H^\beta(\mathbb{R}^N))\cap C^1([0,\infty),L^2(\mathbb{R}^N)), \hbox{such that supp$\varphi\subset Q_T$, $\varphi$ is $\mathbb{R}$-valued}\},$$
and
$$Y_T=\{\varphi\in C([0,\infty),H^\beta(\mathbb{R}^N))\cap C^1([0,\infty),L^\infty(\mathbb{R}^N)), \hbox{such that supp$\varphi\subset Q_T$, $\varphi$ is $\mathbb{R}$-valued}\},$$
where $Q_T:=[0,T]\times\mathbb{R}^N$ and the fractional Sobolev space $H^\beta(\mathbb{R}^N)$ is defined by
$$H^\beta(\mathbb{R}^N)=\{u\in L^2(\mathbb{R}^N); (-\Delta)^{\beta/2}u\in L^2(\mathbb{R}^N)\}.$$

\begin{definition}[$L^2$-weak solution]\label{definitionweak}
Let $u_0\in L^2(\mathbb{R}^N)$ and $T>0.$ We say that $u$ is an
$L^2$-weak solution of \eqref{1} if 
$$u\in L^1((0,T),L^{2}(\mathbb{R}^N))\cap L^p((0,T),L^{2p}(\mathbb{R}^N)),$$
and 
\begin{eqnarray}\label{weaksolution}
&{}&\lambda\int_{Q_T}|u|^p\varphi(t,x)\,dt\,dx+\varepsilon\,i^\alpha\int_{Q_T}u_0(x){}^cD^\alpha_{t|T}\varphi(t,x)\,dt\,dx\nonumber\\
&{}&=\,i^\alpha\int_{Q_T}u\,{}^cD^\alpha_{t|T}\varphi\,dt\,dx-\int_{Q_T}u(-\Delta)^{\beta/2}\varphi(t,x)\,dt\,dx,
\end{eqnarray}
holds for all $\varphi\in X_T$. We denote the lifespan for the $L^2$-weak solution by
$$T_w(\varepsilon):=\sup\{T\in(0,\infty];\,\,\hbox{there exists a unique $L^2$-weak solution u to \eqref{1}}\}.$$
Moreover, if $T>0$ can be arbitrary chosen, i.e. $T_w(\varepsilon)=\infty$, then $u$ is called a global $L^2$-weak solution of \eqref{1}.
\end{definition}

\begin{definition}[$L^1$-weak solution]\label{L1weak}
Let $u_0\in L^1(\mathbb{R}^N)$ and $T>0.$ We say that $u$ is an
$L^1$-weak solution of \eqref{1} if $u,|u|^p\in L^1((0,T),L^1(\mathbb{R}^N))$
and
\begin{eqnarray}\label{weaksolution1}
&{}&\lambda\int_{Q_T}|u|^p\varphi(t,x)\,dt\,dx+\varepsilon\,i^\alpha\int_{Q_T}u_0(x){}^cD^\alpha_{t|T}\varphi(t,x)\,dt\,dx\nonumber\\
&{}&=\,i^\alpha\int_{Q_T}u\,{}^cD^\alpha_{t|T}\varphi\,dt\,dx-\int_{Q_T}u(-\Delta)^{\beta/2}\varphi(t,x)\,dt\,dx,
\end{eqnarray}
holds for all $\varphi\in Y_T$. We denote the lifespan for the $L^1$-weak solution by
$$\overline{T}_w(\varepsilon):=\sup\{T\in(0,\infty];\,\,\hbox{there exists a unique $L^1$-weak solution u to \eqref{1}}\}.$$
Moreover, if $T>0$ can be arbitrary chosen, i.e. $\overline{T}_w(\varepsilon)=\infty$, then $u$ is called a global $L^1$-weak solution to \eqref{1}.
\end{definition}


\section{Preliminaries}\label{sec2}

\begin{definition} [Absolutely continuous functions] \cite[Chapter~1]{SKM}\\
A function $g:[a,b]\rightarrow\mathbb{R}$ with $a,b\in\mathbb{R}$, is absolutely continuous if and only if there exists a Lebesgue summable function $\psi\in L^1(a,b)$ such that 
$$g(t)=g(a)+\int_{a}^t\psi(s)\,ds,\qquad\hbox{for all}\,\, t\in[a,b].$$ 
The space of these functions is denoted by $AC[a,b]$. 
\end{definition}

\begin{definition}[Riemann-Liouville fractional integrals] \cite[Chapter~1]{SKM}\\
Let $g\in L^1(0,T)$ with $T>0$. The Riemann-Liouville left- and right-sided fractional integrals of order $\sigma\in(0,1)$ are, respectively, defined by
$$
I^\sigma_{0|t}g(t):=\frac{1}{\Gamma(\sigma)}\int_{0}^t(t-s)^{-(1-\sigma)}g(s)\,ds, \quad t>0,
$$
and
$$
I^\sigma_{t|T}g(t):=\frac{1}{\Gamma(\sigma)}\int_t^{T}(s-t)^{-(1-\sigma)}g(s)\,ds, \quad t<T,
$$
where $\Gamma$ is the Euler gamma function.
\end{definition}

\begin{definition}[Caputo fractional derivatives] \cite[Chapter~1]{SKM}\\
Let $f\in AC[0,T]$ with $T>0$. The Caputo left- and right-sided fractional derivatives of order $\delta\in(0,1)$ exists almost everywhere on $[0,T]$ and defined, respectively, by
$$
{}^cD^\delta_{0|t}f(t):=\frac{d}{dt}I^{1-\delta}_{0|t}[f(t)-f(0)]=I^{1-\delta}_{0|t}[f^\prime(t)], \quad t>0,
$$
and
$$
{}^cD^\delta_{t|T}f(t):=-\frac{d}{dt}I^{1-\delta}_{t|T}[f(t)-f(T)]=-I^{1-\delta}_{t|T}[f^\prime(t)], \quad t<T.
$$
\end{definition}

\begin{lemma}\cite[Lemma~2.22, p. 96]{KSTr}\\
Let $0<\delta<1$ and $T>0$. If $f\in AC[0,T]$ or $f\in C^1[0,T]$, then 
\begin{equation}\label{28}
  I^\delta_{0|t}\, {}^cD^\delta_{0|t}f(t)=f(t)-f(0).
\end{equation}
\end{lemma}

Given $T>0$, let us define the function $w: [0,T] \to \mathbb{R}$ by the following formula:
\begin{equation}\label{w}
\displaystyle w(t)=\left(1-t/T\right)^{\eta} \quad\hbox{for all}\,\,\,0\leq t\leq T,
\end{equation}
where $\eta \gg1$. Later on, we need the following properties concerning the function $w$.

\begin{lemma}\cite[Property~2.16, p.95]{KSTr}\label{lemma1}\\
Let $T>0$, $\eta>\alpha -1$, and $0<\alpha<1$. For all $t\in[0,T]$, we have
\begin{equation}\label{6}
{}^cD_{t|T}^{\alpha}
w(t)=\frac{\Gamma(\eta+1)}{
\Gamma(\eta+1-\alpha)}T^{-\alpha}(1-t/T)^{\eta-\alpha}.
\end{equation}
\end{lemma}

\begin{lemma}\label{lemma2}
Let $T>0$, $0<\alpha<1$, $\eta>\alpha p/(p-1)-1$, and $p>1$. Then, we have
\begin{equation}\label{8}
\int_0^T (w(t))^{-\frac{1}{p-1}}|{}^cD_{t|T}^{\alpha}
w(t)|^{\frac{p}{p-1}}\,dt=C_1\,T^{1-\alpha\frac{p}{p-1}},
\end{equation}
and
\begin{equation}\label{9}
\int_0^T {}^cD_{t|T}^{\alpha}
w(t)\,dt=C_2\,T^{1-\alpha},
\end{equation}
where
$$C_1=\frac{1}{\eta+1-\alpha\frac{ p}{p-1}}\left[ \frac{\Gamma(\eta+1)}{\Gamma(\eta+1-\alpha)}  \right]^{\frac{p}{p-1}},\quad\hbox{and}\quad C_2= \frac{\Gamma(\eta+1)}{\Gamma(\eta+2-\alpha)}.$$

\end{lemma}

\begin{proof}
Let we start by proving \eqref{8}. Using Lemma \ref{lemma1}, we have
\begin{align*}
\int_{0}^T (w(t))^{-\frac{1}{p-1}}|{}^cD_{t|T}^{\alpha}w(t)|^{\frac{p}{p-1}}\,dt &= \left[ \frac{\Gamma(\eta+1)}{\Gamma(\eta+1-\alpha)}  \right]^{\frac{p}{p-1}}\,T^{-\alpha\frac{p}{p-1}}\int_{0}^T (w(t))^{-\frac{1}{p-1}}(w(t))^{\frac{p({\eta}-\alpha)}{(p-1)\eta}}\,dt\\
&= \left[ \frac{\Gamma(\eta+1)}{\Gamma(\eta+1-\alpha)}  \right]^{\frac{p}{p-1}}\,T^{-\alpha\frac{p}{p-1}}\int_{0}^T(1-t/T)^{\eta-\alpha\frac{ p}{p-1}}\,dt\\
&= \left[ \frac{\Gamma(\eta+1)}{\Gamma(\eta+1-\alpha)}  \right]^{\frac{p}{p-1}}\,T^{1-\alpha\frac{p}{p-1}}\int_{0}^1(1-s)^{\eta-\alpha\frac{ p}{p-1}}\,ds\\
&= C_1\,T^{1-\alpha\frac{p}{p-1}}.
\end{align*}
Similarly, we get \eqref{9}.
\end{proof}

\begin{lemma}\label{lemma6}
Let $T>0$, $0<\alpha<1$, $p>1$, $A,B\geq 0$, and $v\in C^1([0,T),\mathbb{R})$ satisfying the following fractional differential inequality
\begin{equation}\label{27}
 {}^cD_{0|t}^{\alpha}v(t)\geq B\left[|v(t)|^p-A\right], \quad t\in[0,T).
\end{equation}
subject to $v(0)> A^{\frac{1}{p}}$. Then $v(t)\geq A^{\frac{1}{p}}$ for all $t\in[0,T)$.
\end{lemma}
\begin{proof} Fixing $T_1\in(0,T)$, we show that $v(t)\geq A^{\frac{1}{p}}$ for any $t\in(0,T_1]$. Then, since $T_1$ is arbitrary, the claim follows. Let us start by defining $T^*=\inf\{t>0;\,\, v(t)\geq A^{\frac{1}{p}}\}$. Since $v$ is continuous and $v(0)>A^{\frac{1}{p}}$, we have $T^*>0$. We claim $T^*=T_1$. Otherwise, we have $v(t)>A^{\frac{1}{p}}$ for all $t\in(0,T^*)$ such that $v(T^*)=A^{\frac{1}{p}}$; this implies, in particular, that
\begin{equation}\label{29}
F(t,v(t)):=B\left[|v(t)|^p-A\right]\geq 0,\qquad\text{for all}\,\, t\in[0,T^*].
\end{equation}
 On the other hand, since the right hand side of \eqref{27} is continuous on $[0,T_1]$ and $v\in C^1([0,T_1])$, applying the Riemann-Liouville  fractional integral $I^\alpha_{0|T^*}$ to \eqref{27} on $[0,T_1]$ and using \eqref{28}, we get
$$A^{\frac{1}{p}}=v(T^*)=v(0)+\frac{1}{\Gamma(\alpha)}\int_{0}^{T^*}(T^*-s)^{-(1-\alpha)}F(s,v(s))\,ds\geq v(0)>A^{\frac{1}{p}},$$
where we have used \eqref{29}; contradiction. This completes the proof.
\end{proof}

Using \cite[Proposition~4.6]{LIU} and applying the same argument as in the proof of Lemma \ref{lemma6}, one can define the function $g\in C([0,T_b),\mathbb{R}^+)$ which is the unique solution of 
\begin{equation}\label{30}
 \left\{
 \begin{array}{ll}
 {}^cD_{0|t}^{\alpha}g(t)=B\, g^p(t),& \quad t\in[0,T_b),\\\\
 g(0)>0,&\\
 \end{array}
 \right.
\end{equation}
 where $T_b$ is the maximal time of existence. 
\begin{proposition}[Fractional differential inequalities]\label{prop1}${}$\\
 Let $T_b>0$ be the blow-time of the solution of \eqref{30}, and let $T>T_b$, $0<\alpha<1$, $p>1$, $B>0$, and $f\in C^1([0,T),\mathbb{R})$ be a nonnegative solution of the following fractional differential inequality
\begin{equation}\label{57}
 \left\{
 \begin{array}{ll}
 {}^cD_{0|t}^{\alpha}f(t)\geq B\, f^p(t),& \quad t\in[0,T),\\\\
 f(0)>0.&\\
 \end{array}
 \right.
\end{equation}
Then  $f$ blows up at $T_b$, i.e. $\lim_{t\to T_b^{-}}f(t)=+\infty$. Moreover, the following upper and lower bound of $T_b$ are also given
 \begin{equation}\label{39}
  T_L\leq T_b\leq T_U,
  \end{equation}
 where
 $$T_U:=\left(\frac{\Gamma(1+\alpha)}{B\,(f(0))^{p-1}H(p,\alpha)}\right)^{1/\alpha}\qquad \hbox{and}\qquad T_L:=\left(\frac{\Gamma(1+\alpha)}{B\,(f(0))^{p-1}G(p)}\right)^{1/\alpha},$$
with
 \begin{equation}\label{40}
 G(p)=\min\left(2^p,\frac{p^p}{(p-1)^{p-1}}\right),\qquad H(p,\alpha)=\max(p-1,2^{-\frac{p\alpha}{p-1}}).
   \end{equation}
\end{proposition}
\begin{proof} 
 Applying \cite[Theorem~5.1]{Feng}, we conclude that the solution $g$ of \eqref{30} is an increasing function and 
 $$\lim_{t\to T_b^{-}}g(t)=+\infty.$$
 On the other hand, by taking $g(0)=f(0)$, applying \cite[Theorem~4.10]{LIU} and using \eqref{30}, \eqref{57}, we conclude that
 $$f(t)\geq g(t)\geq 0,$$
 this implies that 
  $$\lim_{t\to T_b^{-}}f(t)=+\infty.$$
 Moreover, using  \cite[Theorem~5.2]{Feng}, we get \eqref{39}
\end{proof}

\begin{definition}\cite{Kwanicki,Silvestre}\label{def1}
Let $s \in (0,1)$ and $X$ be a suitable set of functions defined on $\mathbb{R}^N$. The fractional Laplacian $(-\Delta)^s$ in $\mathbb{R}^N$ is a non-local operator defined as the following singular integral
$$ (-\Delta)^s: \,\,v \in X  \mapsto (-\Delta)^s v(x):= C_{N,s}\,\, p.v.\int_{\mathbb{R}^N}\frac{v(x)- v(y)}{|x-y|^{N+2s}}dy, $$
as long as the right-hand side exists, and $p.v.$ stands for Cauchy's principal value, $C_{N,s}:= \frac{4^s \Gamma(\frac{N}{2}+s)}{\pi^{\frac{N}{2}}\Gamma(-s)}$ is a normalization constant and $\Gamma$ denotes the Gamma function.
\end{definition}

\begin{lemma}\cite[Lemma~2.3]{DaoReissig}\label{lemma3}
Let $\langle x\rangle:=(1+|x|^2)^{1/2}$ for all $x\in\mathbb{R}^N$. Let $s \in (0,1)$ and $\phi:\mathbb{R}^N\rightarrow \mathbb{R}$ be a function defined by $\phi(x)=\langle x\rangle^{-q}$, where $n<q\leq N+2s$. Then, $\phi\in H^{2s}(\mathbb{R}^N)$ and the following estimate holds:
\begin{equation}\label{3}
\left|(-\Delta)^s\phi(x)\right|\leq C_{N,q}\phi(x), \quad\hbox{for all}\,\,x\in\mathbb{R}^N,\qquad C_{N,q}=C(s,N,q)>0.
\end{equation}
\end{lemma}

\begin{lemma}\label{lemma4}\cite[Lemma~2.4]{DaoReissig}
Let $s \in (0,1)$, and let Let $\psi$ be a smooth function satisfying $\partial_x^2\psi\in L^\infty(\mathbb{R}^N)$. For any $R>0$, let $\psi_R$ be a function defined by
$$ \psi_R(x):= \psi(x/R) \quad \text{ for all } x \in \mathbb{R}^N.$$
Then, $(-\Delta)^s \psi_R$ satisfies the following scaling properties:
$$(-\Delta)^s \psi_R(x)= R^{-2s}((-\Delta)^s\psi)(x/R), \quad \text{ for all } x \in \mathbb{R}^N. $$
\end{lemma}

\begin{lemma}\label{lemma5}
Let $s \in (0,1)$, $R>0$ and $p>1$. Then, the following estimate holds
$$\int_{\mathbb{R}^N}(\phi_R(x))^{-\frac{1}{p-1}}\,\big|(-\Delta)^s\phi_R(x)\big|^{\frac{p}{p-1}}\, dx\leq C_3 R^{-\frac{2sp}{p-1}+N},$$
where $C_3=(C_{N,q})^{p/(p-1)}\,A_0>0$, $A_0$ is defined below, $\phi_R(x):= \phi({x}/{R})$, and $\phi$ is given in Lemma \ref{lemma3}.
\end{lemma}
\begin{proof} If $0<s<1$, then using the change of variable $\tilde{x}=x/R$ and Lemma \ref{lemma4} we have $(-\Delta)^s\phi_R(x)=R^{-2s}(-\Delta)^s\phi(\tilde{x})$. Therefore, by Lemma \ref{lemma3} we conclude that
$$\int_{\mathbb{R}^N}(\phi_R(x))^{-\frac{1}{p-1}}\,\big|(-\Delta)^s \phi_R(x)\big|^{\frac{p}{p-1}}\, dx\leq (C_{N,q})^{\frac{p}{p-1}} R^{-\frac{2sp}{p-1}+N}\int_{\mathbb{R}^N}\phi(\tilde{x})\, d \tilde{x}= (C_{N,q})^{\frac{p}{p-1}} A_0 R^{-\frac{2sp}{p-1}+N},$$
where
$$A_0=\int_{\mathbb{R}^N}\phi(x)\, d x>0.$$
\end{proof}


\section{Theorem 1. Non-existence of global $L^1$-weak solution in the case $p\leq p_F$}\label{sec3}
To state our first result, we set
$$\lambda=\lambda_1+i\lambda_2,\quad u_0=g+ih,$$
where $\lambda_i\in\mathbb{R}$ ($i=0,1$) and $g$ and $h$ are real-valued functions; the real and imaginary parts of $i^\alpha u_0$ can be written, respectively, as
$$G_1(x)=\cos(\frac{\alpha\pi}{2})g(x)-\sin(\frac{\alpha\pi}{2})h(x),\quad\hbox{and}\quad G_2(x)=\cos(\frac{\alpha\pi}{2})h(x)+\sin(\frac{\alpha\pi}{2})g(x).$$

\begin{theorem}[Non-existence of global $L^1$-weak solution in the case $p\leq p_F$]\label{theorem1}${}$\\
Let $0<\alpha<1$, $0<\beta<2$, $N\geq1$, $\varepsilon=1$.
\begin{enumerate}
\item  If $1<p< 1+\frac{\beta}{N}=p_F$, and $u_0\in L^1(\mathbb{R}^N)$ satisfies
\begin{equation}\label{42}
\lambda_1\int_{\mathbb{R}^N}G_1(x)\,dx>0\quad\text{or}\quad \lambda_2\int_{\mathbb{R}^N}G_2(x)\,dx>0,
\end{equation}
then problem \eqref{1} admits no global $L^1$-weak solution.
\item If $p= p_F$, and $u_0\in L^2(\mathbb{R}^N)$ satisfies
$$|\lambda_1|^{\frac{2-p}{p-1}}\lambda_1\int_{\mathbb{R}^N}G_1(x)\,dx>C_0\,A_0\quad\text{or}\quad  |\lambda_1|^{\frac{2-p}{p-1}}\lambda_2\int_{\mathbb{R}^N}G_2(x)\,dx>C_0\,A_0,$$
where $\displaystyle A_0=\int_{\mathbb{R}^N}\langle x\rangle^{-N-\beta}\, d x$, and $C_0$ is defined in \eqref{11} below, then problem \eqref{1} admits no global $L^1$-weak solution.
\end{enumerate}
\end{theorem}
\proof We argue by contradiction. Suppose that $u$ is a global weak solution to \eqref{1}, then
\begin{eqnarray}\label{2}
&{}&\lambda\int_{Q_T}|u|^p\varphi(t,x)\,dt\,dx+\,i^\alpha\int_{Q_T}u_0(x){}^cD^\alpha_{t|T}\varphi(t,x)\,dt\,dx\nonumber\\
&{}&=\,i^\alpha\int_{Q_T}u{}^cD^\alpha_{t|T}\varphi\,dt\,dx-\int_{Q_T}u(-\Delta)^{\beta/2}\varphi(t,x)\,dt\,dx,
\end{eqnarray}
for all $T>0$ and all $\varphi\in Y_T$. In order to get a non-negativity in the left hand side of \eqref{2}, we consider four cases:\\
Case I: If $\lambda_{1}>0$, then $\int_{\mathbb{R}^N} G_{1}\,dx >0$, therefore by taking the real part (Re) of the both sides of \eqref{2}, we get:
\begin{eqnarray*}
&{}&\lambda_1\int_{Q_T}|u|^p\varphi(t,x)\,dt\,dx+\,\int_{Q_T}G_1(x){}^cD^\alpha_{t|T}\varphi(t,x)\,dt\,dx\\
&{}&=\,\int_{Q_T}\hbox{Re}(i^\alpha\,u){}^cD^\alpha_{t|T}\varphi\,dt\,dx-\int_{Q_T}\hbox{Re}(u)(-\Delta)^{\beta/2}\varphi(t,x)\,dt\,dx.
\end{eqnarray*}
Case II: If $\lambda_{1}<0$, then $\int_{\mathbb{R}^N} G_{1}\,dx <0$ therefore by taking (-Re) of the both sides of  \eqref{2} we get:
\begin{eqnarray*}
&{}&(-\lambda_1)\int_{Q_T}|u|^p\varphi(t,x)\,dt\,dx-\,\int_{Q_T}G_1(x){}^cD^\alpha_{t|T}\varphi(t,x)\,dt\,dx\\
&{}&=-\,\int_{Q_T}\hbox{Re}(i^\alpha\,u){}^cD^\alpha_{t|T}\varphi\,dt\,dx+\int_{Q_T}\hbox{Re}(u)(-\Delta)^{\beta/2}\varphi(t,x)\,dt\,dx.
\end{eqnarray*}
Case III: If $\lambda_{2}>0$, then $\int_{\mathbb{R}^N} G_{2}\,dx >0$, therefore by taking the imaginary part (Im) of the both sides of \eqref{2}, we get:
\begin{eqnarray*}
&{}&\lambda_2\int_{Q_T}|u|^p\varphi(t,x)\,dt\,dx+\,\int_{Q_T}G_2(x){}^cD^\alpha_{t|T}\varphi(t,x)\,dt\,dx\\
&{}&=\,\int_{Q_T}\hbox{Im}(i^\alpha\,u){}^cD^\alpha_{t|T}\varphi\,dt\,dx-\int_{Q_T}\hbox{Im}(u)(-\Delta)^{\beta/2}\varphi(t,x)\,dt\,dx.
\end{eqnarray*}
Case IV: If $\lambda_{2}<0$, then $\int_{\mathbb{R}^N} G_{2}\,dx <0$, therefore by taking (-Im) of the both sides of  \eqref{2}, we get:
\begin{eqnarray*}
&{}&(-\lambda_2)\int_{Q_T}|u|^p\varphi(t,x)\,dt\,dx-\,\int_{Q_T}G_2(x){}^cD^\alpha_{t|T}\varphi(t,x)\,dt\,dx\\
&{}&=-\,\int_{Q_T}\hbox{Im}(i^\alpha\,u){}^cD^\alpha_{t|T}\varphi\,dt\,dx+\int_{Q_T}\hbox{Im}(u)(-\Delta)^{\beta/2}\varphi(t,x)\,dt\,dx.
\end{eqnarray*}
Then we only consider the Case I, since the other cases can be treated in the same way, by assuming $\lambda_{1}>0$,  $u_0\in L^1(\mathbb{R}^N)$ and 
\begin{equation}\label{star}
\int_{\mathbb{R}^N}G_{1}(x)\,dx >0.
\end{equation}
Thus we have
\begin{eqnarray}\label{4}
&{}&\lambda_1\int_{Q_T}|u|^p\varphi(t,x)\,dt\,dx+\,\int_{Q_T}G_1(x){}^cD^\alpha_{t|T}\varphi(t,x)\,dt\,dx\nonumber\\
&{}&\leq\,\int_{Q_T}\left|\cos(\frac{\alpha\pi}{2})\hbox{Re}u-\sin(\frac{\alpha\pi}{2})\hbox{Im}u\right|\left|{}^cD^\alpha_{t|T}\varphi(t,x)\right|\,dt\,dx+\int_{Q_T}\left|\hbox{Re}(u)\right|\left|(-\Delta)^{\beta/2}\varphi(t,x)\right|\,dt\,dx\nonumber\\
&{}&\leq\,2\int_{Q_T}|u|\left|{}^cD^\alpha_{t|T}\varphi(t,x)\right|\,dt\,dx+\int_{Q_T}|u|\left|(-\Delta)^{\beta/2}\varphi(t,x)\right|\,dt\,dx,
\end{eqnarray}
all $\varphi\in Y_T$. Using the $\varepsilon$-Young inequality
\begin{equation}\label{31}
ab\leq \varepsilon\,a^p+\,C_\varepsilon\,b^{\frac{p}{p-1}},\quad\hbox{for all}\,\,\varepsilon>0,\,a,b\geq0,\qquad C_\varepsilon=\frac{(p-1)(p\varepsilon)^{-\frac{1}{p-1}}}{p},
\end{equation}
we get
\begin{eqnarray}\label{5}
&{}&2\int_{Q_T}|u|\left|{}^cD^\alpha_{t|T}\varphi(t,x)\right|\,dt\,dx\nonumber\\
&{}&=\int_{Q_T}|u|\varphi^{1/p}\varphi^{-1/p}2\left|{}^cD^\alpha_{t|T}\varphi(t,x)\right|\,dt\,dx\nonumber\\
&{}&\leq  \varepsilon\int_{Q_T}|u|^p\varphi(t,x)\,dt\,dx+C_4\,\int_{Q_T}\varphi^{-\frac{1}{p-1}}\left|{}^cD^\alpha_{t|T}\varphi(t,x)\right|^{\frac{p}{p-1}}\,dt\,dx,
\end{eqnarray}
where
$$C_4=2^{\frac{p}{p-1}}C_\varepsilon=\frac{2^{\frac{p}{p-1}}(p-1)(p\varepsilon)^{-\frac{1}{p-1}}}{p}.$$
Similarly,
\begin{eqnarray}\label{6}
&{}&\int_{Q_T}|u|\left|(-\Delta)^{\beta/2}\varphi(t,x)\right|\,dt\,dx\nonumber\\
&{}&\leq  \varepsilon\int_{Q_T}|u|^p\varphi(t,x)\,dt\,dx+C_5\,\int_{Q_T}\varphi^{-\frac{1}{p-1}}\left|(-\Delta)^{\beta/2}\varphi(t,x)\right|^{\frac{p}{p-1}}\,dt\,dx,
\end{eqnarray}
where
$$C_5=C_\varepsilon=\frac{(p-1)(p\varepsilon)^{-\frac{1}{p-1}}}{p}.$$
 Combining \eqref{5}-\eqref{6} with \eqref{4}, we obtain
 \begin{eqnarray*}
&{}&(\lambda_1-2\varepsilon)\int_{Q_T}|u|^p\varphi(t,x)\,dt\,dx+\,\int_{Q_T}G_1(x){}^cD^\alpha_{t|T}\varphi(t,x)\,dt\,dx\\
&{}&\leq \,C_4\,\int_{Q_T}\varphi^{-\frac{1}{p-1}}\left|{}^cD^\alpha_{t|T}\varphi(t,x)\right|^{\frac{p}{p-1}}\,dt\,dx+C_5\,\int_{Q_T}\varphi^{-\frac{1}{p-1}}\left|(-\Delta)^{\beta/2}\varphi(t,x)\right|^{\frac{p}{p-1}}\,dt\,dx
\end{eqnarray*}
which implies, by taking $\varepsilon\leq\lambda_1/2$, that
\begin{eqnarray}\label{23}
&{}&\int_{Q_T}G_1(x){}^cD^\alpha_{t|T}\varphi(t,x)\,dt\,dx\nonumber\\
&{}&\leq \,C_4\,\int_{Q_T}\varphi^{-\frac{1}{p-1}}\left|{}^cD^\alpha_{t|T}\varphi(t,x)\right|^{\frac{p}{p-1}}\,dt\,dx+C_5\,\int_{Q_T}\varphi^{-\frac{1}{p-1}}\left|(-\Delta)^{\beta/2}\varphi(t,x)\right|^{\frac{p}{p-1}}\,dt\,dx,
\end{eqnarray}
all $\varphi\in X_T$. At this stage, we take the test function
$$\varphi(t,x):= \phi_R(x) w(t),$$
with $\phi_R(x):=\phi(x/R)$, $R>0$, where $\phi(x)$ and $w(t)$ are defined in Section \ref{sec2} with $s=\beta/2$ and $q=N+\beta$. Therefore, from \eqref{23} we obtain
\begin{eqnarray*}
&{}&\int_{\mathbb{R}^N}G_1(x) \phi_R(x)\,dx \int_0^T {}^cD^\alpha_{t|T}w(t)\,dt\nonumber\\
&{}&\leq \,C_4\,\int_{\mathbb{R}^N}\phi_R(x)\,dx\int_0^T(w(t))^{-\frac{1}{p-1}}\left|{}^cD^\alpha_{t|T}w(t)\right|^{\frac{p}{p-1}}\,dt\nonumber\\
&{}&\quad+\,C_5\,\int_0^Tw(t)\,dt\int_{\mathbb{R}^N}(\phi_R(x))^{-\frac{1}{p-1}}\left|(-\Delta)^{\beta/2}\phi_R(x)\right|^{\frac{p}{p-1}}\,dx.
\end{eqnarray*}
As
$$\int_{\mathbb{R}^N}\phi_R(x)\,dx=\int_{\mathbb{R}^N}\phi(\tilde{x})R^N\,d\tilde{x}=A_0R^N,\quad\hbox{and}\quad \int_0^Tw(t)\,dt=\frac{T}{\eta+1},$$
so, using Lemma \ref{lemma2} and Lemma \ref{lemma5} with $s=\beta/2$ and $\eta>\alpha p/(p-1)-1$, we obtain
$$
C_2\,T^{1-\alpha}\int_{\mathbb{R}^N}G_1(x) \phi_R(x)\,dx \leq\,C_6\,R^N\,T^{1-\alpha\frac{p}{p-1}}+C_7\,T\, R^{-\frac{\beta p}{p-1}+N},
$$
where
$$C_6=C_1\,C_4\,A_0,\qquad\hbox{and}\qquad C_7=\frac{C_3\,C_5}{\eta+1}.$$
Choosing $R=T^{\alpha/\beta}$, we get
\begin{equation}\label{7}
\int_{\mathbb{R}^N}G_1(x) \phi(x/T^{\alpha/\beta})\,dx \leq\,C_8\,T^{\alpha[\frac{N}{\beta}-\frac{1}{p-1}]},
\end{equation}
where
$$
C_8=\frac{1}{C_2}\max\{C_6,C_7\}.
$$
By taking, e.g., $\varepsilon=\lambda_1/2$, $C_8$ can be written as
$$
C_8=C_0\,A_0\,\lambda_1^{-\frac{1}{p-1}},
$$
where
\begin{equation}\label{11}
C_0=\frac{2^{\frac{1}{p-1}}}{p^{\frac{p}{p-1}}C_2}\max\left\{C_1\,2^{\frac{p}{p-1}},\frac{(C_{N,N+\beta})^{\frac{p}{p-1}}}{\eta+1}\right\}.
\end{equation}
If $p< 1+\frac{\beta}{N}$, then $\frac{N}{\beta}-\frac{1}{p-1}<0$. As $G_1\in L^1(\mathbb{R}^N)$, letting $T\rightarrow\infty$ and using the dominated convergence theorem we derive
$$\int_{\mathbb{R}^N}G_1(x)\,dx\leq 0;$$
a contraction with \eqref{star}.\\
If $p=1+\frac{\beta}{N}$, using again the same argument, we arrive at
$$\int_{\mathbb{R}^N}G_1(x)\,dx\leq C_0\,A_0\,\lambda_1^{-\frac{1}{p-1}},$$
which is a contradiction. 

\begin{rmk}
We note that the regularity of $u_0$ is not so important in Theorem \ref{theorem1}, in fact, we can replace  $u_0\in L^1(\mathbb{R}^N)$  by $u_0\in L^2(\mathbb{R}^N)$ and we get a nonexistence of global $L^2$-weak solution. In this case, to ensure the existence of the conditions on $G_1$ and $G_2$, we need also to assume that $G_1$ or $G_2$ are in $ L^1(\mathbb{R}^N)$.
\end{rmk}

\section{Theorem 2. Non-existence of global $L^2$-weak solution in $L^2$-subcritical case for small data}\label{sec4}

\begin{theorem}[Non-existence for global $L^2$-weak solution in $L^2$-subcritical case and for small data]\label{theorem2}${}$\\
Let $0<\alpha<1$, $0<\beta<2$, $N\geq1$, $\varepsilon>0$. Let $u_0\in H^s(\mathbb{R}^N)$, $s\geq0$, and $u$ be an $L^2$-weak solution on $[0,T_w(\varepsilon))$. We assume that $1<p<1+2\beta/N$ and $u_0$ satisfies
\begin{equation}\label{55}
\lambda_1\,G_1(x)\quad\hbox{or}\quad  \lambda_2\,G_2(x)\geq\left\{\begin{array}{ll}
|x|^{-k},&\,\,\hbox{if}\,\,|x|> 1,\\\\
0,&\,\,\hbox{if}\,\,|x|\leq 1,
\end{array}
\right.
\end{equation}
where $N/2<k<\frac{\beta}{p-1}$. Then, $u$ is not global, i.e. $T_w(\varepsilon)<\infty$. More precisely, there exists a constant $\varepsilon_0>0$ such that 
\begin{equation}\label{48}
T_w(\varepsilon)\leq\left\{\begin{array}{ll}
B_0\,\varepsilon^{-\frac{1}{\alpha\kappa_0}},&\,\,\hbox{if}\,\,\varepsilon\in(0,\varepsilon_0),\\\\
1,&\,\,\hbox{if}\,\,\varepsilon\in[\varepsilon_0,\infty),
\end{array}
\right.
\end{equation}
where $\kappa_0=\frac{1}{p-1}-\frac{k}{\beta}>0$ and
$$ B_0=\left(C_0(k+\beta)\omega_N^{-1}\,2^{\frac{N+\beta}{2}}A_0\lambda_1^{\frac{p-2}{p-1}}\right)^{\frac{1}{\alpha\kappa_0}},$$
with $\omega_N$ stands for the $(N-1)$-dimensional surface measure of the unit sphere.
\end{theorem}
\proof Repeating the same calculations as in the proof of Theorem \ref{theorem1}, by taking here $\varepsilon\neq1$, and assuming only
$$\lambda_1>0\quad\hbox{and}\quad G_1(x)\geq\left\{\begin{array}{ll}
\lambda_1^{-1}|x|^{-k},&\,\,\hbox{if}\,\,|x|> 1,\\\\
0,&\,\,\hbox{if}\,\,|x|\leq 1,
\end{array}
\right.
$$
(the other cases can be treated similarly). From \eqref{7}, we obtain 
\begin{equation}\label{19}
\varepsilon \int_{\mathbb{R}^N}G_1(x) \phi(x/T^{\alpha/\beta})\,dx \leq\,C_0\,A_0\,\lambda_1^{-\frac{1}{p-1}}\,T^{\alpha[\frac{N}{\beta}+1-\frac{p}{p-1}]},\qquad\hbox{for all}\,\,0<T<T_w(\varepsilon).
\end{equation}
On the other hand,
\begin{eqnarray*}
\varepsilon\int_{\mathbb{R}^N}G_1(x) \phi(x/T^{\alpha/\beta})\,dx&=&\varepsilon\,T^{\frac{\alpha N}{\beta}}\int_{\mathbb{R}^N}G_1(yT^{\alpha/\beta}) \phi(y)\,dy\\
&\geq& \lambda_1^{-1}\varepsilon\,T^{\frac{\alpha (N-k)}{\beta}}\int_{|y|> T^{-\frac{\alpha N}{\beta}}}|y|^{-k}\phi(y)\,dy\\
&=& \lambda_1^{-1}\varepsilon\,T^{\frac{\alpha (N-k)}{\beta}}K(T),
\end{eqnarray*}
where 
$$K(T):=\int_{|y|> T^{-\frac{\alpha N}{\beta}}}|y|^{-k}\phi(y)\,dy.$$
Therefore, from \eqref{19}, we arrive at
\begin{equation}\label{20}
 \varepsilon\, K(T) \leq\,C_0\,A_0\,\lambda_1^{\frac{p-2}{p-1}}\,T^{\alpha[\frac{k}{\beta}-\frac{1}{p-1}]},\quad\hbox{for all}\,\,0<T<T_w(\varepsilon).
\end{equation}
It remains to estimate from below the last inequality.\\

\noindent First, let $\varepsilon_0=B_0^{\alpha k_0}$, then
$$
T_w(\varepsilon)\leq 1,
$$
for all $\varepsilon\geq \varepsilon_0$. Indeed, suppose on the contrary that there exists $\varepsilon\geq\varepsilon_0$ such that $T_w(\varepsilon)> 1$. Applying \eqref{20} with $\tau\in(1,T_w(\varepsilon))$, we obtain
\begin{equation}\label{21}
 \varepsilon\, K(\tau) \leq\,C_0\,A_0\,\lambda_1^{\frac{p-2}{p-1}}\,\tau^{\alpha[\frac{k}{\beta}-\frac{1}{p-1}]},\quad\hbox{for all}\,\,1<\tau<T_w(\varepsilon).
\end{equation}
Using the fact that
$$|y|\leq(1+|y|^2)^{1/2}\leq \sqrt{2}|y|,\qquad\hbox{for all}\,\,|y|> 1,$$
we have
$$
\frac{\omega_N}{(k+\beta)2^{\frac{N+\beta}{2}}}=2^{-\frac{N+\beta}{2}}\int_{|y|> 1}|y|^{-k-N-\beta}\,dy\leq K(1)\leq\int_{|y|> 1}|y|^{-k-N-\beta}\,dy= \frac{\omega_N}{(k+\beta)}.
$$
Whereupon,
\begin{equation}\label{22}
K(\tau)\geq K(1)\geq \frac{\omega_N}{(k+\beta)2^{\frac{N+\beta}{2}}},\quad\hbox{for all}\,\,1<\tau<T_w(\varepsilon).
\end{equation}
Combining \eqref{21} and \eqref{22}, we obtain
$$
\varepsilon \leq\,(k+\beta)\omega_N^{-1}2^{\frac{N+\beta}{2}}\,C_0\,A_0\,\lambda_1^{\frac{p-2}{p-1}}\,\tau^{\alpha[\frac{k}{\beta}-\frac{1}{p-1}]},
$$
i.e.
$$
\tau\leq B_0\,\varepsilon^{-\frac{1}{\alpha\kappa_0}},\quad\hbox{for all}\,\,1<\tau<T_w(\varepsilon).
$$
Letting $\tau\rightarrow T_w(\varepsilon)$, we get
$$
T_w(\varepsilon)\leq B_0\,\varepsilon^{-\frac{1}{\alpha\kappa_0}}\leq B_0\,\varepsilon_0^{-\frac{1}{\alpha\kappa_0}}=1;
$$
contradiction. Therefore, $T_w(\varepsilon)\leq 1$, for all $\varepsilon\geq \varepsilon_0$.\\

On the other hand, suppose $\varepsilon< \varepsilon_0$. If $T_w(\varepsilon)\leq 1$, it follows that
$$
T_w(\varepsilon)\leq 1= B_0\,\varepsilon_0^{-\frac{1}{\alpha\kappa_0}}\leq B_0\,\varepsilon^{-\frac{1}{\alpha\kappa_0}}.
$$
Hence, it is sufficient to consider $T_w(\varepsilon)> 1$. By the above argument, we get again
$$
T_w(\varepsilon)\leq B_0\,\varepsilon^{-\frac{1}{\alpha\kappa_0}}.
$$
 This completes the proof.\\
${}$\hfill$\square$\\

\begin{rmk}
We note that the condition $k>\frac{N}{2}$ in Theorem \ref{theorem2} is necessary to ensure the existence of at least an $H^s$-function $u_0$ satisfying \eqref{17}, for all $s\geq0$.
\end{rmk}

\section{Theorem 3. Non-existence of global  $L^2$-weak solution for large data}\label{sec5}

\begin{theorem}[Non-existence of global $L^2$-weak solution for $p>1$ and large data]\label{theorem3}${}$\\
Let $0<\alpha<1$, $0<\beta<2$, $N\geq1$, $\varepsilon>0$, and $p>1$. Let $u_0\in H^s(\mathbb{R}^N)$, $s\geq0$, and $u$ be an  $L^2$-weak solution on $[0,T_w(\varepsilon))$. We assume that $u_0$ satisfies
\begin{equation}\label{17}
\lambda_1G_1(x)\quad\hbox{or}\quad  \lambda_2G_2(x)\geq\left\{\begin{array}{ll}
|x|^{-k},&\,\,\hbox{if}\,\,|x|\leq 1,\\\\
0,&\,\,\hbox{if}\,\,|x|> 1,
\end{array}
\right.
\end{equation}
where $k<\min\{\frac{N}{2}-s,\frac{\beta}{p-1}\}$. Then, there exists a constant $\varepsilon_1>0$ such that for any $\varepsilon>\varepsilon_1$, $u$ is not global, i.e. $T_w(\varepsilon)<\infty$. More precisely, 
$$T_w(\varepsilon)\leq \overline{C}\,\varepsilon^{-\frac{1}{\alpha\kappa_0}},$$
for all $\varepsilon>\varepsilon_1$, where $\kappa_0=\frac{1}{p-1}-\frac{k}{\beta}>0$ and
$$ \overline{C}=\left(C_0(N-k)\omega_N^{-1}\,2^{\frac{N+\beta}{2}}A_0\lambda_1^{\frac{p-2}{p-1}}\right)^{\frac{1}{\alpha\kappa_0}}.$$
\end{theorem}
\proof Repeating the same calculations as in the proof of Theorem \ref{theorem1}, by taking here $\varepsilon\neq1$, and considering only the case
$$\lambda_1>0\quad\hbox{and}\quad G_1(x)\geq\left\{\begin{array}{ll}
\lambda_1^{-1}|x|^{-k},&\,\,\hbox{if}\,\,|x|\leq 1,\\\\
0,&\,\,\hbox{if}\,\,|x|> 1,
\end{array}
\right.
$$
as the other cases can be treated similarly. From \eqref{7}, we obtain 
\begin{equation}\label{12}
\varepsilon \int_{\mathbb{R}^N}G_1(x) \phi(x/T^{\alpha/\beta})\,dx \leq\,C_0\,A_0\,\lambda_1^{-\frac{1}{p-1}}\,T^{\alpha[\frac{N}{\beta}+1-\frac{p}{p-1}]},\quad\hbox{for all}\,\,0<T<T_w(\varepsilon).
\end{equation}
On the other hand,
\begin{eqnarray*}
\varepsilon\int_{\mathbb{R}^N}G_1(x) \phi(x/T^{\alpha/\beta})\,dx&=&\varepsilon\,T^{\frac{\alpha N}{\beta}}\int_{\mathbb{R}^N}G_1(yT^{\alpha/\beta}) \phi(y)\,dy\\
&\geq& \lambda_1^{-1}\varepsilon\,T^{\frac{\alpha (N-k)}{\beta}}\int_{|y|\leq T^{-\frac{\alpha N}{\beta}}}|y|^{-k}\phi(y)\,dy\\
&=& \lambda_1^{-1}\varepsilon\,T^{\frac{\alpha (N-k)}{\beta}}L(T),
\end{eqnarray*}
where 
$$L(T):=\int_{|y|\leq T^{-\frac{\alpha N}{\beta}}}|y|^{-k}\phi(y)\,dy.$$
Therefore, from \eqref{12}, we arrive at
\begin{equation}\label{13}
 \varepsilon\, L(T) \leq\,C_0\,A_0\,\lambda_1^{\frac{p-2}{p-1}}\,T^{\alpha[\frac{k}{\beta}-\frac{1}{p-1}]},\quad\hbox{for all}\,\,0<T<T_w(\varepsilon).
\end{equation}
It remains to estimate from below the last inequality.\\

\noindent We claim that there exists a constant $\varepsilon_1>0$ such that for any $\varepsilon>\varepsilon_1$, 
\begin{equation}\label{16}
T_w(\varepsilon)\leq 1.
\end{equation}
Indeed, suppose on the contrary that for all $\varepsilon_1>0$, there exists $\varepsilon>\varepsilon_1$ such that $T_w(\varepsilon)> 1$. Applying \eqref{13} with $T=1$, we have
\begin{equation}\label{14}
\varepsilon\, L(1) \leq\,C_0\,A_0\,\lambda_1^{\frac{p-2}{p-1}}.
\end{equation}
Using the fact that $k<N$, and
$$\frac{1}{2^{\frac{N+\beta}{2}}}\leq \phi(y)\leq 1,\qquad\hbox{for all}\,\,0\leq |y|\leq 1,$$
it is easy to check that
\begin{equation}\label{15}
\frac{\omega_N}{(N-k)2^{\frac{N+\beta}{2}}}\leq L(1)\leq \frac{\omega_N}{(N-k)}.
\end{equation}
Combining \eqref{14} and \eqref{15}, we obtain
$$
\varepsilon \leq\,(N-k)\omega_N^{-1}\,2^{\frac{N+\beta}{2}}\,C_0\,A_0\,\lambda_1^{\frac{p-2}{p-1}}=:\varepsilon_1;
$$
contradiction. Thus the claim is proved.\\

Therefore, for all $T<T_w(\varepsilon)\leq 1$, we have
$$L(T)\geq \int_{|y|\leq 1}|y|^{-k}\phi(y)\,dy=L(1)\geq \frac{\omega_N}{(N-k)2^{\frac{N+\beta}{2}}},$$
which implies, using again \eqref{13},
$$
 \varepsilon\,  \frac{\omega_N}{(N-k)2^{\frac{N+\beta}{2}}} \leq\,C_0\,A_0\,\lambda_1^{\frac{p-2}{p-1}}\,T^{\alpha[\frac{k}{\beta}-\frac{1}{p-1}]},\quad\hbox{for all}\,\,0<T<T_w(\varepsilon),
$$
i.e.
$$T\leq \overline{C}\,\varepsilon^{-\frac{1}{\alpha\kappa_0}},\quad\hbox{for all}\,\,0<T<T_w(\varepsilon).$$
Since $T$ is arbitrary in $(0,T_w(\varepsilon))$, the proof is completed by letting $T\to T_w(\varepsilon)$.\\
${}$\hfill$\square$\\

\begin{rmk}
In Theorem \ref{theorem3}, it is sufficient to just consider the case $p\leq 1+2\beta/(N-2s)$,  because the other case $p>1+2\beta/(N-2s)$ is proved below in Section \ref{sec6}, (non local implies non global existence), and in this case we take $k<\frac{N}{2}-s\,\, (\leq\frac{\beta}{p-1})$.
\end{rmk}

\begin{rmk}
We note that the condition $k<\frac{N}{2}-s$ in Theorem \ref{theorem3} is necessary to ensure the existence of at least an $H^s$-function $u_0$ satisfying \eqref{17}.
\end{rmk}
\section{Theorem 4. Nonexistence of local $L^2$-weak solution in $H^s$-supercritical case}\label{sec6}

\begin{theorem}[Non-existence of local $L^2$-weak solution in $H^s$-supercritical case]\label{theorem4}${}$\\
Let $0<\alpha<1$, $0<\beta<2$, $N\geq1$, $\varepsilon>0$, and $p>1+2\beta/(N-2s)$. Assume $u_0\in H^s(\mathbb{R}^N)$, $0\leq s<N/2$, such that $u_0$ satisfies \eqref{17} with $\beta/(p-1)<k<N/2-s$. Then there is no local $L^2$-weak solution of \eqref{1}.
\end{theorem}
\proof Suppose that there exists an $L^2$-weak solution $u$ on $[0,T)$ for some $0<T<T_w(\varepsilon)$. Repeating the same proof of Theorem \ref{theorem3}, we have
$$
 \varepsilon\, L(\tau) \leq\,C_0\,A_0\,\lambda_1^{\frac{p-2}{p-1}}\,\tau^{\alpha[\frac{k}{\beta}-\frac{1}{p-1}]},\quad\hbox{for all}\,\,0<\tau<T.
$$
For all $\tau< 1$, we have
$$L(\tau)\geq \int_{|y|\leq 1}|y|^{-k}\phi(y)\,dy=L(1)\geq \frac{\omega_N}{(N-k)2^{\frac{N+\beta}{2}}},$$
whereupon
$$
 \varepsilon\,  \frac{\omega_N}{(N-k)2^{\frac{N+\beta}{2}}} \leq\,C_0\,A_0\,\lambda_1^{\frac{p-2}{p-1}}\,\tau^{\alpha[\frac{k}{\beta}-\frac{1}{p-1}]},\quad\hbox{for all}\,\,0<\tau<\min\{1,T\},
$$
i.e.
$$
 \varepsilon\,   \leq(N-k)\omega_N^{-1}\,2^{\frac{N+\beta}{2}}\,C_0\,A_0\,\lambda_1^{\frac{p-2}{p-1}}\,\tau^{\alpha[\frac{k}{\beta}-\frac{1}{p-1}]},\quad\hbox{for all}\,\,0<\tau<\min\{1,T\}.
$$
As $\beta/(p-1)<k$, we have $k/\beta-1/(p-1)>0$. Therefore, taking $\tau\rightarrow 0^+$, we obtain $\varepsilon=0$; contradiction. This completes the proof.\\
${}$\hfill$\square$\\

\section{Theorem 5. Nonexistence of local $L^1$-weak solution  in the case $p>p_F$}\label{sec7}

\begin{theorem}[Non-existence of local $L^1$-weak solution in the supercritical case]\label{theorem5}${}$\\
Let $0<\alpha<1$, $0<\beta<2$, $N\geq1$, $\varepsilon>0$, and $p>1+\beta/N=p_F$. Assume $u_0\in L^1(\mathbb{R}^N)$ and satisfying \eqref{17} with $\beta/(p-1)<k<N$. Then there is no local $L^1$-weak solution of \eqref{1}.
\end{theorem}
\proof  Suppose that there exists an $L^1$-weak solution $u$ on $[0,T)$ for some $0<T<\overline{T}_w(\varepsilon)$. Applying the proof of Theorem \ref{theorem4}, step by step. The only difference is the condition $k<N$ instead of $k<N/2-s$, which is required to ensure that there exists an $L^1$-function $u_0$ satisfying \eqref{17}.

\section{Theorem 6. Nonexistence of global $L^2$-weak solution: New approach}\label{sec8}

\begin{theorem}[Nonexistence for global $L^2$-weak solution: New approach]\label{theorem6}${}$\\
Let $0<\alpha<1$, $0<\beta<2$, $N\geq1$, $p>1$, $T>0$, and
$$X(T)=C([0,T),L^2(\mathbb{R}^N))\cap C^1([0,T),H^{-\frac{\beta}{2}}(\mathbb{R}^N))\cap L^\infty((0,T),L^{p}(\mathbb{R}^N)).$$
Assume $u_0\in L^2(\mathbb{R}^N)$ and satisfies
\begin{equation}\label{36}
M_R(0)>C_{N,p,\beta,\gamma}R^{N-\frac{\beta}{p-1}},
\end{equation}
for some $R>0$ and $\gamma\in\mathbb{C}$ satisfying $\text{Re}(\gamma\lambda)>0$,
where
$$M_R(0)=\text{Re}\left(i^\alpha\gamma\int_{\mathbb{R}^N}u(0,x)\phi_R(x)\,dx\right),$$
with $\phi_R(x):=\phi(x/R)$, $R>0$, ($\phi(x)$ is defined in Section \ref{sec2} with $q=N+\beta$),
and
$$
C^p_{N,p,\beta,\lambda,\gamma}=2\,C_{1/2}(\text{Re}\left(\gamma \lambda\right))^{-\frac{p}{p-1}} |\gamma|^{^{\frac{p^2}{p-1}}} A_0^{p}\left(C_{N,N+\beta}\right)^{\frac{p}{p-1}}.
$$
Then there is no distributional solution $u\in X(T)$, with $T>T_b$, for \eqref{1}, where (see \eqref{39})
\begin{equation}\label{44}
T_b\sim \left(\frac{R^{N(p-1)}\Gamma(1+\alpha)}{D_{N,p,\beta,\lambda,\gamma}\,\left[M_R(0)-C_{N,p,\beta,\gamma}R^{N-\frac{\beta}{p-1}}\right]^{p-1}}\right)^{1/\alpha},
\end{equation}
and
$$D_{N,p,\beta,\lambda,\gamma}=2^{-1}\,\text{Re}\left(\gamma \lambda\right) |\gamma|^{-p} A_0^{-(p-1)}.$$
\end{theorem}
\begin{proof} Suppose, on the contrary, that there exists a distributional solution $u\in X(T)$ with $T>T_b$. Let 
$$M_R(t)=\text{Re}\left(i^\alpha\gamma\int_{\mathbb{R}^N}u(t,x)\phi_R(x)\,dx\right).$$
By Lemmas \ref{lemma3} and \ref{lemma4}, we have
\begin{eqnarray}\label{32}
{}^cD^\alpha_{t|T}M_R(t)&=&\text{Re}\left(\gamma\int_{\mathbb{R}^N}i^\alpha\,{}^cD^\alpha_{t|T} u(t,x)\phi_R(x)\,dx\right)\nonumber\\
&=&\text{Re}\left(\gamma \lambda\right)\int_{\mathbb{R}^N}|u(t,x)|^p\phi_R(x)\,dx +\,R^{-\beta}\text{Re}\left(\gamma\int_{\mathbb{R}^N}u(t,x)\left((-\Delta)^{\beta/2}\phi\right)(x/R)\,dx\right)\nonumber\\
&\geq&\text{Re}\left(\gamma \lambda\right)\int_{\mathbb{R}^N}|u(t,x)|^p\phi(x/R)\,dx-\,C_{N,N+\beta}\,R^{-\beta}|\gamma|\int_{\mathbb{R}^N}|u(t,x)|\phi(x/R)\,dx.
\end{eqnarray}
In order to get a differential inequality, we start by estimating the second term in the right hand side of \eqref{32}. Using $1/2$-Young's inequality \eqref{31}, we obtain
\begin{eqnarray}\label{33}
&{}&C_{N,N+\beta}\,R^{-\beta}|\gamma|\int_{\mathbb{R}^N}|u(t,x)|\phi(x/R)\,dx\nonumber\\
&{}&=\int_{\mathbb{R}^N}|u(t,x)|\left[\text{Re}\left(\gamma \lambda\right)\phi(x/R)\right]^{1/p}C_{N,N+\beta}\,R^{-\beta}|\gamma|\left[\text{Re}\left(\gamma \lambda\right)\right]^{-1/p}\left[\phi(x/R)\right]^{(p-1)/p}\,dx\nonumber\\
&{}&\leq\frac{1}{2}\,\text{Re}\left(\gamma \lambda\right)\int_{\mathbb{R}^N}|u(t,x)|^p\phi(x/R)\,dx\nonumber\\
&{}&\quad+\,C_{1/2}|\gamma|^{\frac{p}{p-1}}\left(C_{N,N+\beta}\right)^{\frac{p}{p-1}}R^{-\frac{\beta p}{p-1}}\left(\text{Re}\left(\gamma \lambda\right)\right)^{-\frac{1}{p-1}}\int_{\mathbb{R}^N}\phi(x/R)\,dx\nonumber\\
&{}&=\frac{1}{2}\,\text{Re}\left(\gamma \lambda\right)\int_{\mathbb{R}^N}|u(t,x)|^p\phi(x/R)\,dx\nonumber\\
&{}&\quad+\,C_{1/2}\,A_0|\gamma|^{\frac{p}{p-1}}\left(\text{Re}\left(\gamma \lambda\right)\right)^{-\frac{1}{p-1}}\left(C_{N,N+\beta}\right)^{\frac{p}{p-1}}R^{N-\frac{\beta p}{p-1}},
\end{eqnarray}
where
$$C_{1/2}=(p-1)p^{-\frac{p}{p-1}}2^{\frac{1}{p-1}}\quad\hbox{and}\quad A_0=\int_{\mathbb{R}^N}\phi(\tilde{x})\,d\tilde{x}.$$
On the other hand, by estimating the first term in the right hand side of \eqref{32}  by using H\"{o}lder's inequality, we get
\begin{eqnarray*}
|M_R(t)|&=&\left|\text{Re}\left(i^\alpha\gamma\int_{\mathbb{R}^N}u(t,x)\phi(x/R)\,dx\right)\right|\\
&\leq&|\gamma|\int_{\mathbb{R}^N}|u(t,x)|\phi(x/R)\,dx\\
&=&|\gamma|\int_{\mathbb{R}^N}|u(t,x)|(\phi(x/R))^{\frac{1}{p}}(\phi(x/R))^{\frac{p-1}{p}}\,dx\\
&\leq&|\gamma|\left(\int_{\mathbb{R}^N}|u(t,x)|^p\phi(x/R)\,dx\right)^{\frac{1}{p}}\left(\int_{\mathbb{R}^N}\phi(x/R)\,dx\right)^{\frac{p-1}{p}}\\
&=&|\gamma| A_0^{\frac{p-1}{p}}R^{\frac{N(p-1)}{p}}\left(\int_{\mathbb{R}^N}|u(t,x)|^p\phi(x/R)\,dx\right)^{\frac{1}{p}},
\end{eqnarray*}
i.e. 
\begin{equation}\label{34}
\int_{\mathbb{R}^N}|u(t,x)|^p\phi(x/R)\,dx\geq |\gamma|^{-p} A_0^{-(p-1)}R^{-N(p-1)}|M_R(t)|^p.
\end{equation}
Inserting \eqref{33}-\eqref{34} into \eqref{32}, we conclude that
\begin{eqnarray*}
{}^cD^\alpha_{t|T}M_R(t)&\geq&2^{-1}\,\text{Re}\left(\gamma \lambda\right)\int_{\mathbb{R}^N}|u(t,x)|^p\phi(x/R)\,dx\nonumber\\
&{}&-\,C_{1/2}\,A_0|\gamma|^{\frac{p}{p-1}}\left(\text{Re}\left(\gamma \lambda\right)\right)^{-\frac{1}{p-1}}\left(C_{N,N+\beta}\right)^{\frac{p}{p-1}}R^{N-\frac{\beta p}{p-1}}\\
&\geq&2^{-1}\,\text{Re}\left(\gamma \lambda\right) |\gamma|^{-p} A_0^{-(p-1)}R^{-N(p-1)}|M_R(t)|^p\\
&{}&-\,C_{1/2}\,A_0|\gamma|^{\frac{p}{p-1}}\left(\text{Re}\left(\gamma \lambda\right)\right)^{-\frac{1}{p-1}}\left(C_{N,N+\beta}\right)^{\frac{p}{p-1}}R^{N-\frac{\beta p}{p-1}}\\
&=&2^{-1}\,\text{Re}\left(\gamma \lambda\right) |\gamma|^{-p} A_0^{-(p-1)}R^{-N(p-1)}\left[|M_R(t)|^p-C^p_{N,p,\beta,\lambda,\gamma}R^{p(N-\frac{\beta}{p-1})}\right],
\end{eqnarray*}
i.e.
\begin{equation}\label{35}
{}^cD^\alpha_{t|T}M_R(t)\geq D_{N,p,\beta,\lambda,\gamma}R^{-N(p-1)}\left[|M_R(t)|^p-C^p_{N,p,\beta,\gamma}R^{p(N-\frac{\beta}{p-1})}\right].
\end{equation}
Applying Lemma \ref{lemma6} and using \eqref{36}, we conclude that
\begin{equation}\label{37}
M_R(t)\geq C_{N,p,\beta,\gamma}R^{N-\frac{\beta}{p-1}}>0,\qquad\text{for all}\,\, t\in[0,T),
\end{equation}
which implies, by using \eqref{35} and  the following elementary inequality 
$$a^p-b^p\geq (a-b)^p,\qquad \text{for all}\,\,a>b\geq 0,\,p>1,$$
that 
\begin{equation}\label{38}
{}^cD^\alpha_{t|T}M_R(t)\geq D_{N,p,\beta,\lambda,\gamma}R^{-N(p-1)}\left[M_R(t)-C_{N,p,\beta,\gamma}R^{N-\frac{\beta}{p-1}}\right]^p.
\end{equation}
Apply Proposition \ref{prop1} and the fact that ${}^cD^\alpha_{t|T} C=0$, for any constant $C>0$, we infer that 
$$\lim_{t\rightarrow T_b}M_R(t)=+\infty.$$
Since
$$M_R(t)\leq \|u(t)\|_{L^\infty((0,T),L^2(\mathbb{R}^N))}\|\phi(\cdotp/R)\|_{L^2(\mathbb{R}^N)}<\infty,\quad\text{for all}\,\,t\in[0,T),$$
we get a contradiction, and this completes the proof.
\end{proof}

\begin{rmk}
Note that, from \eqref{40}, we have $H(p,\alpha)=\max(p-1,2^{-\frac{p\alpha}{p-1}})\geq p-1$; this implies that
$T_b$ can be chosen as
$$T_b =\left(\frac{R^{N(p-1)}\Gamma(1+\alpha)}{(p-1)D_{N,p,\beta,\lambda,\gamma}\,\left[M_R(0)-C_{N,p,\beta,\gamma}R^{N-\frac{\beta}{p-1}}\right]^{p-1}}\right)^{1/\alpha},$$
which is the same blow-up time as in the ordinary differential equation when $\alpha=1$.
\end{rmk}

\begin{corollary}[Theorem~1: New approach]\label{coro1}${}$\\
Let $0<\alpha<1$, $0<\beta<2$, $N\geq1$, $\gamma\in\mathbb{C}$, $\varepsilon=1$, $p>1$. Assume that $p<1+\beta/N$, and $u_0\in L^1(\mathbb{R}^N)\cap L^2(\mathbb{R}^N)$ satisfies
\begin{equation}\label{41}
\text{Re}(\gamma\lambda)>0\qquad\text{and}\qquad\text{Re}\left(i^\alpha\gamma\int_{\mathbb{R}^N}u_0(x)\,dx\right)>0.
\end{equation}
Then there is no distributional solution $u\in X(T)$ to \eqref{1} for sufficiently large $T>0$.
\end{corollary}
\begin{proof}  
By \eqref{41}, using the dominated convergence theorem, we conclude that
$$\lim_{R\to\infty}M_R(0)=\text{Re}\left(i^\alpha\gamma\int_{\mathbb{R}^N}u_0(x)\,dx\right)>0.$$
On the other hand, as $p<1+\beta/N$, 
$$C_{N,p,\beta,\gamma}R^{N-\frac{\beta}{p-1}}\longrightarrow0,\qquad\text{when}\,\, R\to\infty.$$
Therefore, there exists $R_0>0$ such that condition \eqref{36} is satisfied. Using Theorem \ref{theorem6}, the proof is completed.
\end{proof}
\begin{rmk}
Note that, by taking $\gamma=\pm 1, \pm i$ in Corollary \ref{coro1},  condition \eqref{41} implies \eqref{42}, which means that  \eqref{41} is more general that  \eqref{42}. Therefore, in the subcritical case, Theorem \ref{theorem1} can be seen as a particular case of Corollary \ref{coro1}, but with different regularity.
\end{rmk}

\begin{corollary}[Theorem~2: New approach]\label{coro2}${}$\\
Let $0<\alpha<1$, $0<\beta<2$, $N\geq1$, $\varepsilon>0$, $\gamma\in\mathbb{C}$, and $p>1$. Assume that $p<1+2\beta/N$, and $u_0\in H^s(\mathbb{R}^N)$, $s\geq0$, satisfies
\begin{equation}\label{43}
\text{Re}(\gamma\lambda)>0\qquad\text{and}\qquad\text{Re}\left(i^\alpha\gamma u_0(x)\right)\geq\left\{\begin{array}{ll}
|x|^{-k},&\,\,\hbox{if}\,\,|x|> 1,\\\\
0,&\,\,\hbox{if}\,\,|x|\leq 1,
\end{array}
\right.
\end{equation}
where $N/2<k<\frac{\beta}{p-1}$. Then, there exists a constant $\varepsilon_2>0$ such that for all $\varepsilon\in(0,\varepsilon_2]$, there is no distributional solution $u\in X(T)$ to \eqref{1} for sufficiently large $T>T_b$ with $T_b$ defined in \eqref{44}. Moreover $T_b$ can be estimated as follows
\begin{equation}\label{47}
T_b\leq B_1\,\varepsilon^{-\frac{1}{\alpha\kappa_1}},
\end{equation}
for all $\varepsilon\in(0,\varepsilon_2]$, where $\kappa_1=\frac{1}{p-1}-\frac{\min(N,k)}{\beta}>0$,
$$B_1=(p-1)^{-1/\alpha} D_{N,p,\beta,\lambda,\gamma}^{-1/\alpha}\Gamma(1+\alpha)^{1/\alpha}\,2^{\frac{1}{\alpha\kappa_1}}\,(C_{N,p,\beta,\gamma})^{\frac{\min(N,k)(p-1)}{\alpha\beta\kappa_1}}\,I_1^{-\frac{1}{\alpha\kappa_1}},$$
and
$$I_1:=\left\{\begin{array}{ll}
2^{-N-\beta-1}\omega_N(N-k)^{-1}R^{N-k},&\,\,\hbox{if}\,\,k<N,\\\\
\displaystyle 2^{-N-\beta}\omega_N\int_1^2r^{N-1-k}\,dr,&\,\,\hbox{if}\,\,k\geq N.
\end{array}
\right.
$$
\end{corollary}
\begin{proof} In order to apply Theorem \ref{theorem6}, we need to estimate $M_R(0)$ from below, for some $R>0$. Let
$$\varepsilon_2=\left\{\begin{array}{ll}
I_1^{-1}C_{N,p,\beta,\gamma}2^{1-\frac{\beta\kappa_1}{N-k}},&\,\,\hbox{if}\,\,k<N,\\\\
 I_1^{-1}C_{N,p,\beta,\gamma}2^{1-\beta\kappa_1},&\,\,\hbox{if}\,\,k\geq N.
\end{array}
\right.
$$
Let $\varepsilon\in(0,\varepsilon_2]$. We choose $R=R(\varepsilon)$ such that
\begin{equation}\label{50}
\left\{\begin{array}{ll}
R\geq 2^{1/(N-k)},&\,\,\hbox{if}\,\,k<N,\\\\
R\geq 2,&\,\,\hbox{if}\,\,k\geq N.
\end{array}
\right.
\end{equation}
Then, as $R^{N-k}-1\geq R^{N-k}/2$, when $k<N$, using \eqref{43}, we have
 \begin{eqnarray*}
M_R(0)&\geq&\varepsilon\,\text{Re}\left(\gamma\, i^{\alpha}\int_{\mathbb{R}^N}u_0(x)\phi(x/R)\,dx\right)\\
&\geq& \varepsilon \int_{|x|\geq1}|x|^{-k}\phi(x/R)\,dx\\
&\geq& \varepsilon \int_{1\leq|x|\leq R}|x|^{-k}\phi(x/R)\,dx\\
&\geq& \varepsilon 2^{-N-\beta} \int_{1\leq|x|\leq R}|x|^{-k}\,dx\\
&=& \varepsilon 2^{-N-\beta} \omega_N\int_1^ R r^{N-1-k}\,dr\\
&\geq& \varepsilon 2^{-N-\beta} \omega_N\left\{\begin{array}{ll}
(N-k)^{-1}(R^{N-k}-1),&\,\,\hbox{if}\,\,k<N,\\\\
\displaystyle \int_1^ 2 r^{N-1-k}\,dr,&\,\,\hbox{if}\,\,k\geq N,
\end{array}
\right.\\
&\geq& \varepsilon\,I_1\,R^{(N-k)_+},
\end{eqnarray*}
with $(N-k)_+=\max(N-k,0)$. Therefore
\begin{equation}\label{45}
M_R(0)-C_{N,p,\beta,\gamma}R^{N-\frac{\beta}{p-1}}\geq R^{(N-k)_+}\left(\varepsilon\,I_1-C_{N,p,\beta,\gamma}R^{-\beta\kappa_1}\right)= R^{(N-k)_+}\left(\frac{\varepsilon\,I_1}{2}\right)>0,
\end{equation}
where $R$ is chosen to ensure the last equality, namely
\begin{equation}\label{46}
R=\left(\frac{2C_{N,p,\beta,\gamma}}{\varepsilon\,I_1}\right)^{\frac{1}{\beta\kappa_1}}.
\end{equation}
It is clear, by our choice of $\varepsilon_2$, that condition \eqref{50} is satisfied. Applying Theorem \ref{theorem6}, we conclude that there is no solution $u\in X(T)$ to \eqref{1} for all $T>T_b$. Moreover, from \eqref{39},\eqref{44}  and the fact that $ H(p,\alpha)\geq p-1$, we obtain
$$T_b\leq T_U\leq(p-1)^{-1/\alpha}  \left(\frac{R^{N(p-1)}\Gamma(1+\alpha)}{D_{N,p,\beta,\lambda,\gamma}\,\left[M_R(0)-C_{N,p,\beta,\gamma}R^{N-\frac{\beta}{p-1}}\right]^{p-1}}\right)^{1/\alpha}.$$
Then, using \eqref{45} and \eqref{46}, we conclude that
$$T_b\leq B_1\,\varepsilon^{-\frac{1}{\alpha\kappa_1}}.$$
 This complete the proof.
\end{proof}
\begin{rmk}
Note that, $\kappa_1>\kappa_0$, this means that \eqref{48} is better than \eqref{47}. Moreover, by taking $\gamma=\pm 1, \pm i$ in Corollary \ref{coro2},  condition \eqref{43} implies \eqref{55}, which means that  \eqref{43} is more general that  \eqref{55}. Therefore, Theorem \ref{theorem2} can be seen as a particular case of Corollary \ref{coro2}, but with different regularity.
\end{rmk}

\begin{corollary}[Theorem~3: New approach]\label{coro3}${}$\\
Let $0<\alpha<1$, $0<\beta<2$, $N\geq1$, $\varepsilon>0$, $\gamma\in\mathbb{C}$, and $p>1$. Assume that $u_0\in H^s(\mathbb{R}^N)$, $s\geq0$, satisfies
\begin{equation}\label{56}
\text{Re}(\gamma\lambda)>0\qquad\text{and}\qquad\text{Re}\left(i^\alpha\gamma u_0(x)\right)\geq\left\{\begin{array}{ll}
|x|^{-k},&\,\,\hbox{if}\,\,|x|\leq 1,\\\\
0,&\,\,\hbox{if}\,\,|x|> 1,
\end{array}
\right.
\end{equation}
where $k<\min\{\frac{N}{2}-s,\frac{\beta}{p-1}\}$. Then, there exists a constant $\varepsilon_3>0$ such that for any $\varepsilon\geq\varepsilon_3$, there is no distributional solution $u\in X(T)$ to \eqref{1} for sufficiently large $T>T_b$ with $T_b$ is defined in \eqref{44}. Moreover $T_b$ can be estimated as follows
\begin{equation}\label{}
T_b\leq B_2\,\varepsilon^{-\frac{1}{\alpha\kappa_0}},
\end{equation}
for all $\varepsilon\geq\varepsilon_3$, where $\kappa_0=\frac{1}{p-1}-\frac{k}{\beta}>0$,
$$B_2=(p-1)^{-1/\alpha} D_{N,p,\beta,\lambda,\gamma}^{-1/\alpha}\Gamma(1+\alpha)^{1/\alpha}\,2^{\frac{1}{\alpha\kappa_0}}\,(C_{N,p,\beta,\gamma})^{\frac{k(p-1)}{\alpha\beta\kappa_0}}\,I_2^{-\frac{1}{\alpha\kappa_2}},$$
and
$$I_2:=2^{-N-\beta}\omega_N(N-k)^{-1}.$$
\end{corollary}
\begin{proof} In order to apply Theorem \ref{theorem6}, we need to estimate $M_R(0)$ from below, for some $R>0$. Let
$$\varepsilon_3=2\,I_2^{-1}C_{N,p,\beta,\gamma}.$$
Let $\varepsilon\geq\varepsilon_3$. We choose $R=R(\varepsilon)\leq 1$. Then, using \eqref{56}, we have
 \begin{eqnarray*}
M_R(0)&\geq&\varepsilon\,\text{Re}\left(\gamma\, i^{\alpha}\int_{\mathbb{R}^N}u_0(x)\phi(x/R)\,dx\right)\\
&\geq& \varepsilon \int_{|x|\leq1}|x|^{-k}\phi(x/R)\,dx\\
&\geq& \varepsilon \int_{|x|\leq R}|x|^{-k}\phi(x/R)\,dx\\
&\geq& \varepsilon 2^{-N-\beta} \int_{|x|\leq R}|x|^{-k}\,dx\\
&=& \varepsilon 2^{-N-\beta} \omega_N\int_0^ R r^{N-1-k}\,dr\\
&=& \varepsilon 2^{-N-\beta} \omega_N(N-k)^{-1}R^{N-k}\\
&=& \varepsilon \,I_2\,R^{N-k}.
\end{eqnarray*}
Therefore
\begin{equation}\label{52}
M_R(0)-C_{N,p,\beta,\gamma}R^{N-\frac{\beta}{p-1}}\geq R^{N-k}\left(\varepsilon\,I_2-C_{N,p,\beta,\gamma}R^{-\beta\kappa_0}\right)= R^{N-k}\left(\frac{\varepsilon\,I_2}{2}\right)>0,
\end{equation}
where 
\begin{equation}\label{51}
R=\left(\frac{2C_{N,p,\beta,\gamma}}{\varepsilon\,I_2}\right)^{\frac{1}{\beta\kappa_0}}.
\end{equation}
It is clear, by our choice of $\varepsilon_3$, that $R\leq 1$. Applying Theorem \ref{theorem6}, we conclude that there is no solution $u\in X(T)$ of \eqref{1} for all $T>T_b$. Moreover, from \eqref{39},\eqref{44}  and $H(p,\alpha)\geq p-1$, we obtain
$$T_b\leq T_U\leq(p-1)^{-1/\alpha}  \left(\frac{R^{N(p-1)}\Gamma(1+\alpha)}{D_{N,p,\beta,\lambda,\gamma}\,\left[M_R(0)-C_{N,p,\beta,\gamma}R^{N-\frac{\beta}{p-1}}\right]^{p-1}}\right)^{1/\alpha}.$$
Then, using \eqref{52} and \eqref{51}, we conclude that
$$T_b\leq B_2\,\varepsilon^{-\frac{1}{\alpha\kappa_0}}.$$
 This complete the proof.
\end{proof}
\begin{rmk}
Note that, by taking $\gamma=\pm 1, \pm i$ in Corollary \ref{coro3},  condition \eqref{56} implies \eqref{17}, which means that  \eqref{56} is more general that  \eqref{17}. Therefore, Theorem \ref{theorem3} can be seen as a particular case of Corollary \ref{coro3}, but with different regularity.
\end{rmk}


\bibliographystyle{elsarticle-num}



\end{document}